\documentclass{amsart}
\usepackage{Loop-Fusion,caption} 
\title{Loop-fusion cohomology and transgression}
\author{Chris Kottke}
\address{Northeastern University\\Department of Mathematics}
\email{c.kottke@neu.edu}
\author{Richard Melrose}
\address{MIT\\Department of Mathematics}
\email{rbm@math.mit.edu}
\thanks{The second author received partial support during the
  preparation of this paper under NSF grant  DMS-1005944.}
\begin{document}
\begin{abstract} 
`Loop-fusion cohomology' is defined on the continuous loop space of a manifold
in terms of \v Cech cochains satisfying two multiplicative conditions with
respect to the fusion and figure-of-eight products on loops. The main result is
that these cohomology groups, with coefficients in an abelian group, are
isomorphic to those of the manifold and the transgression homomorphism factors
through the isomorphism.
\end{abstract}
\maketitle

In this note we present a refined \v Cech cohomology of the continuous free
loop space $\cL M$ of a manifold $M$ (or we could work throughout with the
energy space instead). Compared to the standard theory, the cochains are
limited by multiplicativity conditions under two products on loops, the fusion
product (defined by Stolz and Teichner \cite{Stolz-Teichner2005}) and the
figure-of-eight product (which appears implicitly in Barrett \cite{barrett} and
explicitly in \cite{Kottke-Melrose-FLS}). The main result of this paper is that
the resulting `loop-fusion' cohomology, $\vHffe^\bullet(\cL M; A),$ recovers
the cohomology of the manifold directly on the loop space.

\begin{figure}[h]
\begin{tikzpicture}
\tikzset{->-/.style={decoration={ markings, mark=at position #1 with {\arrow{to}}},postaction={decorate}}} 
\matrix (m) [matrix of nodes, column sep=2cm]
{(a) \begin{tikzpicture}[baseline=(current bounding box.north), thick, scale=0.5,rotate=-30]%[rotate=-30]
\fill (-2,-2) circle (4pt);
\fill (2,2) circle (4pt);
\draw[->-=.5] (-2,-2) .. controls (-0.3,0.3) and (0.3,-0.3)  .. (2,2);
\draw[->-=.5] (-2,-2)  .. controls (-1.5,-1) and (-1,0) .. (-1,1) %node[above left] {$\gamma_3$}
.. controls (-1,2) and (1,3).. (1,2) .. controls (1,1) and (2,2).. (2,2);
\draw[->-=.5] (-2,-2)  .. controls (-2,-2) and (0.2,-2.2) .. (1.2,-1.2) %node[below] {$\gamma_1$}
.. controls (2.2,-0.2) and (2,2).. (2,2);
\end{tikzpicture}
&
(b) \begin{tikzpicture}[baseline=(current bounding box.north), thick, scale=0.5, rotate=-25]%[rotate=-25]
\fill (-2,-2) circle (4pt);
\fill (0,0) circle (4pt);
\fill (2,2) circle (4pt);
\draw[->-=.5] (-2,-2) .. controls (-1,-2) and (0,-1)
.. %node[below] {$\gamma_1$}
(0,0) ;
\draw[->-=.5] (0,0) .. controls (1,0) and (2,1)
.. %node[below] {$\gamma'_1$}
(2,2) ;
\draw[->-=.5] (-2,-2) .. controls (-2,-1) and (-1,0).. %node[above] {$\gamma_2$}
(0,0) ;
\draw[->-=.5] (0,0)
.. controls (0,1) and (1,2) .. %node[above] {$\gamma'_2$}
(2,2);
%\draw (-2,-2) -- (2,2);
\end{tikzpicture} \\};
\end{tikzpicture}
\caption{Fusion (a) and figure-of-eight (b) configurations.}
\label{F:figure}
\end{figure}
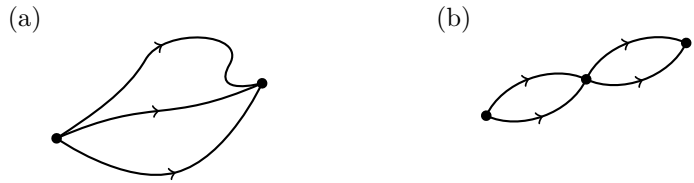

\begin{thmstar} For each $k \geq 1$ and discrete abelian group $A$ there is an 
enhanced transgression isomorphism
\[
	\lfT : \vH^k(M; A) \stackrel \cong \to \vHffe^{k-1}(\cL M; A),
\]
forming a commutative diagram with the forgetful map, $f,$ to ordinary
cohomology and the standard transgression map $T:$
\begin{equation}
\begin{tikzpicture}[baseline=(current bounding box.center), ->,>=to,auto]
\matrix (m) [matrix of math nodes, column sep=1cm, row sep=1cm, text height=2ex, text depth=0.25ex]
{\vH^k(M;A) &  \vH^{k-1}_{\ffe}(\cL M;A) \\ & \vH^{k-1}(\cL M;A).
\\};
\path (m-1-1) edge node {$\lfT$} (m-1-2); %top
\path (m-1-2) edge node {$f$} (m-2-2); %right
%\path (m-2-1) edge node {$\varepsilon$} (m-2-2); %bot
\path (m-1-1) edge node[below] {$T$} (m-2-2); %diag
\end{tikzpicture}
\label{E:TRDiag}
\end{equation}
\label{T:ffe_isomorphism}
\end{thmstar}
\noindent For $A = \bbZ$ and $k=2$ or $k=3$ this result
appears in \cite{Kottke-Melrose-FLS}. There the cohomology
classes are represented geometrically by functions and circle bundles over the
loop space which satisfy the fusion property and are reparameterization
equivariant; the figure-of-eight condition follows from these conditions.

The case $k=2$ with integer coefficients is closely related to the problem of
recovering a circle bundle on $M$ up to isomorphism from its holonomy as a
function on $\cL M$, which has been considered by Teleman \cite{teleman},
Barrett \cite{barrett} and Caetano-Picken \cite{caetano}. In \cite{Waldorf-I},
Waldorf considers principal bundles for general abelian groups and makes explicit use of
the fusion product. The case $k = 3$ corresponds to an association between
gerbes on $M$ and circle bundles on $\cL M.$ Such a construction was first
given by Brylinski \cite{Brylinski3}, and in \cite{Brylinski-McLaughlin},
Brylinski and McLaughlin point out that the resulting bundle on the smooth loop
space has an action by $\Diff(\bbS)$ and is multiplicative with respect to the
composition of loops based at the same point. In \cite{Waldorf-II},
\cite{Waldorf-III} Waldorf identifies the fusion property for bundles on $\cL
M$ given by the transgression of gerbes, and uses this to define an inverse
functor.

The extension of such results to $k\ge3$ to give an explicit transgression
of geometric objects, such as higher gerbes, faces the usual obstacles
associated with compatibility conditions. Here, the use of \v Cech
cohomology allows for a short and unified treatment of the general case. In
particular this shows that the two conditions included in the loop-fusion structure,
without equivariance with respect to the variable on the circle or thin
homotopy equivalence, suffice to capture the cohomology of $M.$

\section{Spaces, covers and \v Cech cohomology} \label{S:covers}
%%%%%%%%%%%%%%%%%%%%%%%%%%%%%%%%%%%%%%%%%%%%%%%%%%%%%%%%%%%%%%%%%%%%%%%%%%%%%%%%%%%%%%%%%
\subsection{Base space} \label{S:base}
Let $M$ be a compact smooth manifold. In the subsequent discussion we fix
a Riemann metric on $M$ and $\epsilon >0$ smaller than the injectivity
radius although refinement arguments show that none of the results depend
on these choices. For each $m \in M$ let $U_m$ be the open geodesic ball of
radius $\epsilon > 0$ centered at $m$ and consider the disjoint union of
these balls as a cover of $M:$ 
\[
	\cU = \bigsqcup_{m \in M} U_m \to M.
\]
This is a good cover: for $k\ge1,$ each of the $k$-fold intersections is
empty or contractible. The disjoint union of these intersections is
equivalent to the fiber product
\[
	\cU^{(k)} = \cU\times_{M} \cdots \times_{M} \cU \to M.
\]

\begin{rmk}\label{R:Rmk.1} It is convenient to work with `maximal' covers parameterized by
  the space itself. However it is possible throughout the discussion below
  to restrict to countable covers as is more conventional in \v Cech
  theory. Indeed, one can work here with the cover of $M$ by neighborhoods
  with centers at a countable dense subset. See the subsequent remark on
  paths and loops.

Also, though we have assumed $M$ to be smooth and compact for convenience, the
result we present applies to a wider variety of spaces. Indeed, we only use
that $M$ has a good cover, with respect to which there are compatible good
covers of the path and loop spaces as below. 
\end{rmk}

The collection $\set{M^n : n \geq 1}$ forms a simplicial space with the
projections $\pi_i : M^n \to M^{n-1}$, $1 \leq i \leq n$, as face maps with the
convention that $\pi_i$ omits the $i$th factor. Similarly $\set{\cU^n : n \geq
1}$ is a simplicial space, with face maps also denoted $\pi_i$; each $\cU^n \to
M^n$ is also a good cover. Differentials deriving from this simplicial
structure will be denoted by $\pa.$

For each fixed $n$ the successive fiber products $\set{(\cU^{(k)})^n \equiv
  (\cU^n)^{(k)}: k \geq 1}$ also form a simplicial space with face maps
$\iota_j : (\cU^n)^{(k+1)} \to (\cU^n)^{(k)}$ the inclusions of
$(k+1)$-fold intersections of the open sets into the $k$-fold
intersections. This second simplicial space underlies the \v Cech cohomology of
$M^n$. Indeed, for an abelian group $A$ the \v Cech cochains on $M^n$ with
respect to $\cU^n$ are the locally constant maps
\[
	\vC^k(M^n; A) \ni \alpha : (\cU^n)^{(k+1)} \to A, \quad k \in \bbN
\]
with differential
\begin{equation}
\begin{gathered}
	\delta : \vC^{k}(M^n; A) \to \vC^{k+1}(M^n; A), \\
	\delta \alpha = \prod_{j=1}^{k+2} \iota_j^* f^{(-1)^j} : (\cU^n)^{(k+2)} \to A.
\end{gathered}
	\label{E:cech_differential}
\end{equation}
Note that these are {\em unoriented} \v Cech cochains, so that $\alpha$ is not
required to be odd with respect to permutations acting on the fiberwise factors
of $\cU^{(k)} \to M$.

For a good cover such as $\cU^n,$ the \v Cech cohomology is isomorphic to
the ordinary cohomology of $M^n$ \cite{godement}:
\[
\vH^\bullet(M^n; A) := H^\bullet(\vC^\bullet(M^n; A), \delta) \cong H^\bullet(M^n; A).
\]
%where $\underline A$ is the sheaf of continuous functions into $A$.

\begin{lem} For each $k$, the sequence
\[
\begin{gathered}
	\vH^k(M; A) \stackrel \pa \to \vH^k(M^2; A) \stackrel \pa \to \vH^k(M^3; A) \stackrel \pa \to \cdots \\
	\pa : \vH^k(M^n; A) \ni \alpha \to
\prod_{j=1}^{n+1} \pi_j^\ast \alpha^{(-1)^j} \in \vH^k(M^{n+1}; A) 
\end{gathered}
\]
is exact.
\label{L:ex_cohom_M}
\end{lem}
\begin{proof}
The same computation as for the \v Cech differential \eqref{E:cech_differential}
shows that $\pa^2 = 0$. Fix a point $\bar m \in M$ and consider the
inclusions 
\[
	i_n : M^n \hookrightarrow M^{n+1}, \quad
	(m_1,\ldots, m_n) \mapsto (\bar m, m_1,\ldots, m_n).
\]
Then
\[
	\pi_j \circ i_n = \begin{cases} \id & j = 1, \\ i_{n-1} \circ \pi_{j-1} & j \geq 2, \end{cases}
\]
as maps from $M^n$ to $M^n$ and for $\alpha \in\vH^{k}(M^n;A),$ 
\[
	i_n^*\pa \alpha = \prod_{j=1}^{n+1} i_n^*\pi_j^*\alpha^{(-1)^j} 
	%= \alpha^{-1}\prod_{j=1}^n \pi_j^*i_{n-1}^*\alpha^{(-1)^{j+1}} 
	= \alpha^{-1}\,\pns{\pa i_{n-1}^*\alpha^{-1}}.
	%\pa i_{n-1}^*\alpha = \prod_{j = 1}^{n} \pi_{j}^*i_n^* \alpha^{(-1)^j} =\alpha\,\pns{i_n^*\pa\alpha}
\]
Thus if $\pa\alpha =1$ then $\alpha =\pa i_{n-1}^*\alpha^{-1}.$
\end{proof}

%%%%%%%%%%%%%%%%%%%%%%%%%%%%%%%%%%%%%%%%%%%%%%%%%%%%%%%%%%%%%%%%%%%%%%%%%%%%%%%%%%%%%%%%%
\subsection{Path space} \label{S:path}
Let $\cI M = \cC([0,1]; M)$ be the free continuous path space of $M;$ it is
a Banach manifold which fibers over $M^2$ by the endpoint map
\[
\varepsilon : \cI M \to M^2, \ \varepsilon(\gamma) = \big(\gamma(0),\gamma(1)\big).
\]
We make use of the {\em join product} 
\begin{equation}
\begin{gathered}
	j : \pi_{3}^* \cI M \times_{M^3} \pi_{1}^* \cI M \to \pi_{2}^* \cI M \\
	j(\gamma_1,\gamma_2)(t) = \begin{cases} \gamma_1(2t) & 0 \leq t \leq 1/2, \\ \gamma_2\big(2(t - 1/2)\big) & 1/2 \leq t \leq 1, \end{cases}
\end{gathered}
	\label{E:join}
\end{equation}
where $(\gamma_1,\gamma_2) \in \pi_3^* \cI M \times_{M^3} \pi_1^* \cI M$ if and
only if $\gamma_1(1) = \gamma_2(0)$. Note that \eqref{E:join} is a
bijection and hence $\pi_3^*\cI M \times_{M^3} \pi_1^* \cI M$ can be
identified with $\cI M$ fibering over $M^3$ by the 
map $\gamma \mapsto (\gamma(0),\gamma(1/2),\gamma(1)).$

For $\gamma \in \cI M$, let $\Gamma_\gamma = \set{\gamma' \in \cI M : \sup_t
\abs{\gamma(t)-\gamma'(t)} < \epsilon}$ be the set of paths lying pointwise
within the metric tube of radius $\epsilon$ around $\gamma.$ Proceeding as
above and setting
\[
	\Gamma = \bigsqcup_{\gamma \in \cI M} \Gamma_\gamma \to \cI M, \quad
	\Gamma^{(k)} = \Gamma \times_{\cI M} \cdots \times_{\cI M} \Gamma,
\]
gives a good cover of $\cI M$, which factors through $\cU^2$, i.e.\ the
diagram
\begin{equation}
\begin{tikzpicture}[baseline=(current bounding box.center), ->,>=to,auto]
\matrix (m) [matrix of math nodes, column sep=1cm, row sep=1cm, text height=2ex, text depth=0.25ex]
{\Gamma^{(k)}  & (\cU^2)^{(k)} \\ \cI M & M^2 \\};
\path (m-1-1) edge node {$\varepsilon$} (m-1-2); %top
\path (m-1-2) edge (m-2-2); %right
\path (m-2-1) edge node {$\varepsilon$} (m-2-2); %bot
\path (m-1-1) edge (m-2-1); %left
\end{tikzpicture}
	\label{E:compat_covers}
\end{equation}
commutes for each $k$. Furthermore, join lifts to a well-defined map
\begin{equation}
j : \pi_{3}^* \Gamma^{(k)} \times_{(\cU^3)^{(k)}} \pi_{1}^* \Gamma^{(k)} \to \pi_{2}^* \Gamma^{(k)},
	\label{E:lift_join}
\end{equation}
and there is a natural identification of $\pi_3^* \Gamma^{(k)}
\times_{(\cU^3)^{(k)}} \pi_1^* \Gamma^{(k)}$ with $\Gamma^{(k)}$.

\begin{rmk}\label{R:Rmk.2} 
As noted in Remark~\ref{R:Rmk.1} above, it is possible to work throughout with
countable covers. One can restrict to neighborhoods centered on paths which are
finite combinations of segments with rational end-points and which are affine
geodesics between the chosen countable dense set in the manifold. The resulting
cover has the crucial property of being closed under join, and the induced
countable cover of loop space, considered below, is closed with respect to the
two loop-fusion operations. It is also possible to work over a more general space, 
provided $M$ and $\cI M$ have good covers satisfying \eqref{E:compat_covers} and \eqref{E:lift_join}.
\end{rmk}

The definition of the \v Cech cochain complex above carries over to $\cI M$ (finite
dimensionality of $M$ was not used there) giving
\[
	\vC^k(\cI M; A) \ni f : \Gamma^{(k+1)} \to A, 
	\quad \delta f = \prod_{j=1}^{k+2} \iota_j^* f^{(-1)^j} \in \vC^{k+1}(\cI M; A),
\]
where we reuse the notation $\iota_j : \Gamma^{(k+1)} \to \Gamma^{(k)}$ for the
face maps of the simplicial space $\set{\Gamma^{(k)} ; k \geq 1},$ and observe
that again $\vH^k(\cI M; A) \cong H^k(\cI M; A)$ since $\Gamma$ is a good cover.

The identification of $\pi_3^*\Gamma\times_{\cU^3}\pi_1^*\Gamma$ with
$\Gamma$ and \eqref{E:lift_join} gives a second chain map on $\vC^\bullet(\cI M
; A)$ associated to the simplicial structure on $\set{M^n : n \geq 1}$:
\[
\bar \pa : \vC^k(\cI M; A) \to \vC^k(\cI M; A), \quad
\bar \pa f = \pi_3^* f^{-1}\,\pi_1^* f^{-1}\,j^*(\pi_2^* f) : \Gamma^{(k)} \to A.
\]
This does not lead to a complex, i.e.\ $\bar \pa^2$ is not trivial, since $\cI
M$ is not itself a simplicial space over $\set{M^n : n \geq 1}$; 
reparameterization is required to compare pullbacks of paths.

The constant paths may be identified as an inclusion $M \subset \cI M.$ Let
\[
	\vC_0^k(\cI M; A)  = \set{f \in \vC^k(\cI M; A) : f \rst_{M} = 1}
\]
denote the subcomplex of cochains which are trivial on them. Since the join
map restricts to the trivial map on constant paths $\bar \pa : \vC_0^\bullet(\cI
M; A) \to \vC_0^\bullet(\cI M ; A).$

\begin{lem}
The subcomplex $(\vC_0^\bullet(\cI M; A),\delta)$ is acyclic.
\label{L:path_subcomplex}
\end{lem}

\begin{proof} The short exact sequence of chain complexes
\[
	0 \to \vC_0^\bullet(\cI M; A) \to \vC^\bullet(\cI M; A) \to \vC^\bullet(M; A) \to 0
\]
induces a long exact sequence in cohomology, however $\vH^\bullet(\cI M; A)
\cong \vH^\bullet(M; A)$ since there is a deformation retraction of $\cI M$
onto $M$, from which it follows that $\vH_0^\bullet(\cI M; A) = 0$.
\end{proof}

%%%%%%%%%%%%%%%%%%%%%%%%%%%%%%%%%%%%%%%%%%%%%%%%%%%%%%%%%%%%%%%%%%%%%%%%%%%%%%%%%%%%%%%%%
\subsection{Loop space} \label{S:loop}
For $l \geq 1$ we denote by $\cI^{[l]} M$ the fiber product
\[
	\cI^{[l]} M = \cI M \times_{M^2} \cdots \times_{M^2} \cI M,
\]
and observe that $\cI^{[2]} M = \set{(\gamma_1,\gamma_2) : \gamma_1(t) =
\gamma_2(t), \ t = 0,1}$ may be identified with the Banach manifold of free continuous
loops by {\em fusion} of paths:
\[
\psi : \cI^{[2]} M \stackrel \cong \to \cL M = \cC(\bbS; M), \quad 
\ell(t) = \psi(\gamma_1, \gamma_2)(t) = \begin{cases} \gamma_1(t), & 0 \leq t \leq 1 \\
\gamma_2(-t), & -1 \leq t \leq 0 \end{cases}
\]
where $\bbS$ is parameterized as $[-1,1]/(\set{-1}\sim\set{1})$ for later convenience.

The set $\set{\cI^{[l]} M : l \geq 1}$ forms another simplicial space, with face maps
given by the fiber projections $\varrho_j : \cI^{[l]} M \to \cI^{[l-1]} M$, $1
\leq j \leq l,$ and $\set{\Gamma^{[l]} : l \geq 1}$ forms a good cover, where
\[
	\Gamma^{[l]} = \Gamma \times_{\cU^2} \cdots \times_{\cU^2} \Gamma \to \cI^{[l]} M,
\]
is lifted from the path space with $k$-fold overlaps
\[
(\Gamma^{(k)})^{[l]} \equiv (\Gamma^{[l]})^{(k)} = \Gamma^{[l]} \times_{\cI^{[l]} M} \cdots \times_{\cI^{[l]} M} \Gamma^{[l]}.
\]
For clarity of notation, we denote this cover of loop space by
\[
	\Lambda = \Gamma^{[2]} \to \cL M.
\]
We will denote differentials derived from this simplicial space or
its cover by $d$.  

Passing to $\cI^{[l]} M$ in \eqref{E:join} gives rise to a map
\begin{equation}
	j^{[l]} : \pi_3^* \cI^{[l]} M \times_{M^3} \pi_1^* \cI^{[l]} M \to \pi_2^* \cI^{[l]} M,
	\label{E:join_l}
\end{equation}
and its local version
\begin{equation}
	j^{[l]} : \pi_3^* (\Gamma^{[l]})^{(k)} \times_{(\cU^3)^{(k)}} \pi_1^* (\Gamma^{[l]})^{(k)} 
	\to \pi_2^*(\Gamma^{[l]})^{(k)}.
	\label{E:join_l_local}
\end{equation}

In the case $l = 2$, we call this the {\em figure-of-eight product} on loops as
in \cite{Kottke-Melrose-FLS}. The product of two loops $\ell_1 =
\psi(\gamma_{11},\gamma_{12})$ and $\ell_2 = \psi(\gamma_{21},\gamma_{22})$
such that $\ell_1(1) = \ell_2(0)$ is the loop $\ell_3 =
\psi\big(j(\gamma_{11},\gamma_{21}), j(\gamma_{12},\gamma_{22})\big).$ See Figure~\ref{F:figure}.(b).
The domain in \eqref{E:join_l} with $l = 2$ may be identified with the subspace of {\em
figure-of-eight loops} in $M:$
\[
\cL_8 M = \set{\ell \in \cL M : \ell(1/2) = \ell(-1/2)} \longrightarrow M^3.
\]
This Banach manifold fibers over $M^3$ and has a good cover given by the
domain in \eqref{E:join_l_local} with $l = 2$ and $k = 1.$ Unlike the case
$l = 1$, $\cL_8M$ cannot be identified with the full loop space nor is $j^{[2]}$ invertible.

There is another product on loop space, considered already in
\cite{Stolz-Teichner2005}, associated to $\cI^{[3]}M.$ If
$(\gamma_1,\gamma_2,\gamma_3) \in \cI^{[3]} M$, then $\ell_3 =
\psi(\gamma_1,\gamma_3)$ is the {\em fusion product} (on loops) of $\ell_1 =
\psi(\gamma_1,\gamma_2)$ and $\ell_2 = \psi(\gamma_2,\gamma_3).$ See Figure~\ref{F:figure}.(a).

Within the \v Cech cochain complex $\big(\vC^\bullet(\cL M ;A),\delta\big)$ for loop space:
\[
	\vC^k(\cL M; A) \ni f : \Lambda^{(k+1)} \to A, \quad \delta f = \prod_{j=1}^{k+2} \iota_j^* f^{(-1)^j} \in \vC^{k+1}(\cL M; A),
\]
consider the subcomplex of {\em fusion cochains}
\[
\begin{gathered}
	\vCfus^k(\cL M; A) = \set{ f \in \vC^k(\cL M; A) : df = 1} \\
	df = \varrho_1^*f^{-1}\,\varrho_2^* f\,\varrho_3^* f^{-1} \in \vC^k(\cI^{[3]} M; A).
\end{gathered}
\]
Note that $d^2 : \vC^k(\cI^{[l]}; A) \to \vC^k(\cI^{[l+2]} M; A)$ is trivial
and $\delta d = d\delta$ so this is indeed a subcomplex.

The subspace $\cL_8M \subset \cL M$ is closed under fusion so
$\vCfus^\bullet(\cL_8M; A)$ is well-defined, and imposing a
condition over the figure-of-eight product leads to the \emph{loop-fusion}
subcomplex
\begin{equation}
\begin{gathered}
\vCffe^k(\cL M; A) = \set{f \in \vCfus^k(\cL M ; A) : \bar \pa f = \delta
  g\text{ for }g \in \vCfus^{k-1}(\cL_8 M; A)}, \\
\bar \pa f = \pi_3^*f^{-1}\,\pi_1^* f^{-1}\,(j^{[2]})^*(\pi_2^* f) \in \vC^k_\fus(\cL_8 M; A).
\end{gathered}
	\label{E:ffe_chains}
\end{equation}
Thus, this complex consists of those fusion cochains which are multiplicative
with respect to the figure-of-eight product up to a fusion boundary. The image
of $\bar \pa$ on these chains lies in the space of fusion \v Cech cochains on
the space of figure-of-eight loops; though we do not need to consider it here,
$\bar \pa^2$ may be sensibly defined (it is not automatically trivial).
That
\eqref{E:ffe_chains} is a subcomplex follows from the fact that $\delta \bar \pa = \bar \pa \delta.$
It is also the case that $d \bar \pa = \bar \pa d$ on suitably defined spaces, in
particular as maps from $\vC^k(\cI M; A)$ to $\vC^{k+1}(\cI M; A)$ and from $\vC^k(\cL
M; A)$ to $\vC^{k+1}(\cL_8 M; A).$ 

The {\em loop-fusion} cohomology of $\cL M$ is then defined to be
\begin{equation}
\vHffe^\bullet(\cL M; A) = H^\bullet\big(\vCffe^\bullet(\cL M; A), \delta\big) \to
\vH^\bullet(\cL M ; A),
\label{E:forget_ffe}
\end{equation}
with its homomorphism, $f,$ to ordinary \v Cech cohomology induced by the inclusion of
$\vCffe^\bullet(\cL M; A)$ in $\vC^\bullet(\cL M; A).$

\section{Transgression and Regression} \label{S:trans_regr}
We proceed to the proof of the Theorem above.
\subsection{Transgression} \label{S:transgression}
We first construct the map $\lfT.$ Let $\alpha \in \vC^k(M; A)$ be a cocycle for
$k \geq 1$, and consider
\[
	\varepsilon^* \pa \alpha \in \vC_0^k(\cI M; A),\quad
	\pa \alpha = \pi_1^* \alpha^{-1}\pi_2^*\alpha \in \vC^k(M^2; A).
\]
Since $\delta \varepsilon^*\pa\alpha = \varepsilon^* \pa \delta \alpha = 1$ and
$\vC_0^\bullet(\cI M; A)$ is exact by Lemma~\ref{L:path_subcomplex}, it follows
that $\varepsilon^* \pa \alpha = \delta \beta$ for some $\beta \in
\vC_0^{k-1}(\cI M; A);$ set
\begin{equation}
\omega = d \beta = \varrho_1^* \beta^{-1}\,\varrho_2^* \beta \in \vC^{k-1}(\cL M; A).
\label{E:PathLift}
\end{equation}
Then $\varepsilon \circ \varrho_1 = \varepsilon \circ \varrho_2$ implies
\[
\delta\omega = d \delta \beta =
\varrho_1^*(\varepsilon^*\pa \alpha)^{-1}\,\varrho_2^*(\varepsilon^*\pa \alpha) = 1.
\]
Moreover $d^2 = 1$ so
\[
	d\omega = d^2 \beta = 1 \implies \omega \in \vCfus^{k-1}(\cL M; A).
\]
Finally, $\omega$ is fusion-figure-of-eight since $\bar \pa \omega = d\bar\pa \beta$ and 
$\bar \pa \beta$, which lies in $\vC_0^k(\cI M; A)$ by Lemma~\ref{L:path_subcomplex},
is a boundary. Indeed, for any path $\gamma = j(\gamma_1,\gamma_2)$, 
\begin{multline*}
\delta\bar \pa \beta(\gamma) = \bar\pa \varepsilon^*\pa \alpha(\gamma) =
\varepsilon^*\pa\alpha^{-1} (\gamma_1)\,\varepsilon^*\pa\alpha^{-1}(\gamma_2)\,
\varepsilon^*\pa\alpha(\gamma)
\\= \alpha\big(\gamma_1(0)\big)\alpha^{-1}\big(\gamma_1(1)\big)\alpha
\big(\gamma_2(0)\big)\alpha^{-1}(\gamma_2(1))\alpha^{-1}
\big(\gamma(0)\big)\alpha\big(\gamma(1)\big) = 1.
\end{multline*}
Thus $\bar \pa \beta$ is a cocycle and as $\vC_0^\bullet(\cI M; A)$ is acyclic
there exists $\eta \in \vC_0^{k-2}(\cI M; A)$ such that $\bar \pa
\beta = \delta \eta$. It follows that
\[
	\bar \pa \omega = d\bar\pa \beta = d\delta \eta = \delta d \eta, \quad
	d(d\eta) = 1 \implies \omega \in \vCffe^{k-1}(\cL M; A).
\]

Consider next the effect of the choices made. If $\beta' \in \vC_0^{k-1}(\cI M;
A)$ is another cochain such that $\delta \beta' = \varepsilon^*\pa \alpha$,
then $\delta (\beta' \beta^{-1}) = 1$ implies that $\beta' = \beta \delta\nu$
for some $\nu \in \vC_0^{k-2}(\cI M; A)$, which alters $\omega$ by the boundary
term $\delta d \nu$.  Similarly if $\alpha' = \alpha\delta\mu$ is another
representative for $[\alpha] \in \vH^k(M; A)$, it follows that $\omega' =
\omega \delta \sigma$, where $\sigma$ is the result of the same construction
applied to $\mu$. Thus the {\em transgression map}
\begin{equation}
	\lfT : \vH^k(M; A) \to \vHffe^{k-1}(\cL M; A), \quad \lfT[\alpha] = [\omega]^{-1}
	\label{E:transgression}
\end{equation}
is well-defined.

\subsection{Regression} \label{S:regression}
Next we define a map which is shown below to be the inverse of $\lfT.$ Suppose
$\omega \in \vCffe^{k-1}(\cL M; A)$ is a cocycle, so
\[
	\delta \omega = 1, \ d \omega = 1, \ \bar \pa \omega = \delta \nu, \ d \nu = 1.
\]
Then $\omega$ gives {\em descent data} for the trivial principal $A$-bundle
\begin{equation}
	\Gamma^{(k)} \times A \to \Gamma^{(k)}
	\label{E:trivial_Gamma}
\end{equation}
over $(\cU^2)^{(k)}.$ That is, multiplication by $\omega$ determines a
relation on the fibers, with the content of $d\omega =1$ being that this is
an equivalence relation so inducing a well-defined principal $A$-bundle
$P_k \to (\cU^2)^{(k)}$:
\[
\begin{gathered}
(P_k)_{(m,m')} = \set{(\gamma,a) \in \Gamma^{(k)}\times A : \varepsilon(\gamma)
= (m,m')} / \sim_\omega \\
(\gamma,a) \sim_\omega (\gamma',a') \iff a = \omega(\gamma,\gamma') a'.
\end{gathered}
\]

The condition $\delta \omega = 1$ implies that $P_k$ is a simplicial bundle
(see \cite{Brylinski-McLaughlin}, \cite{murray-stevenson}), i.e.\ the bundle
over $(\cU^2)^{(k+1)}$ consisting of the alternating tensor products of the
pullbacks of $P_k$ by the maps $\iota_j : (\cU^2)^{(k+1)} \to (\cU^2)^{(k)}$ is
canonically trivial:
\[
\delta P_k = \bigotimes_{j} \iota_j^\ast P_k^{(-1)^j} \cong (\cU^2)^{(k+1)}\times A \to
(\cU^2)^{(k+1)}.
\]
Similarly, $\nu$ determines a principal $A$-bundle
\[
	R_{k-1} = \Gamma^{(k-1)}\times A /\sim_\nu\to (\cU^3)^{(k-1)},
\]
and by functoriality of descent there is a canonical isomorphism
\begin{equation}
	\pa P_k \cong \delta R_{k-1} \to (\cU^3)^{(k)},\quad
	\pa P_k = \pi_1^*P_k^{-1}\otimes\pi_2^* P_k\otimes \pi_3^*P_k^{-1}.
	\label{E:iso_descent_bundles}
\end{equation}

The components of $(\cU^2)^{(k)}$ and $(\cU^3)^{(k-1)}$ are contractible so
there exist sections 
\[
	s : (\cU^2)^{(k)} \to P_k, \ \text{and}\  r : (\cU^3)^{(k-1)} \to R_{k-1}.
\]
These pull back to give sections $\delta s$ of $\delta P_k$ and $\delta r$ of
$\delta R_{k-1}$ and as $\delta P_k$ is canonically trivial $\delta s$
gives rise to a cocycle
\[
	\kappa = \delta s \in \vC^k(M^2; A), \ \delta \kappa = \delta \delta s = 1,
\]
where $\delta^2 s$ coincides with the canonical trivialization of $\delta^2 P$ for
any section $s.$ Another choice of section $s'$
alters $\kappa$ by a term $\delta\gamma $, where $\gamma \in \vC^{k-1}(M^2; A)$ is
fixed by $s' = s\,\gamma .$ Thus $[\kappa] \in \vH^k(M^2; A)$ is
determined by $\omega.$ Similarly, another choice $\omega'$ such that $\omega' =
\omega \delta \mu$, $d\mu = 1$ leads to a bundle $P'_k$ and a canonical
isomorphism $P'_k \cong P_k \otimes \delta Q_{k-1}$, where $Q_{k-1}$ is formed
by descent using $\mu.$ If $\kappa = \delta s$ and $\kappa' = \delta s'$ for
respective sections $s$ and $s'$ of $P_k$ and $P'_k,$ if $q$ is any section of
$Q_{k-1}$, and $s' = (s\otimes \delta q)\,\nu$ for some $\nu \in \vC^{k-1}(M^2;
A)$, then $\kappa' = \kappa \delta^2 q\,\delta \nu = \kappa\delta \nu.$
Thus the map from $\vHffe^{k-1}(\cL M;A)$ to $\vH^k(M^2; A)$ is well-defined.

Finally, we may compare $\pa s$ and $\delta r$ as sections of
\eqref{E:iso_descent_bundles}; let $\tau \in \vC^{k-1}(M^3; A)$ 
be determined by $\pa s = \delta r\, \tau$,
from which it follows that
\[
	\pa \kappa = \delta(\pa s) = \delta^2 r\,\delta \tau = \delta \tau \in \vC^{k}(M^3; A).
\]
(A different choice of $r$ leads to $\pa \kappa = \delta \tau'$ for some other $\tau' \in \vC^{k-1}(M^3; A).$)
Thus $\pa [\kappa] = 1 \in \vH^k(M^3; A)$ and so by Lemma~\ref{L:ex_cohom_M},
$[\kappa] = \pa[\alpha]$ for a unique class $[\alpha] \in \vH^k(M; A).$ It
follows that the {\em regression map} is well-defined by
\begin{equation}
	R : \vHffe^{k-1}(\cL M; A) \to \vH^k(M; A), \quad R[\omega] = [\alpha]^{-1}.
	\label{E:regression}
\end{equation}

\begin{prop}
The maps \eqref{E:transgression} and \eqref{E:regression} are inverses.
\label{P:inverses}
\end{prop}
\begin{proof}
To see that $\lfT R = \id$ fix a cocycle $\omega \in \vCffe^{k-1}(\cL M; A)$ and
let $\alpha \in \vC^k(M; A)$ represent $R[\omega]^{-1}$, so that $\pa \alpha =
\kappa\delta \nu$ for some $\nu \in \vC^{k-1}(M^2; A)$, where $\kappa = \delta
s \in \vC^{k}(M^2; A)$ for a choice of section $s$ of the bundle $P_k$.
Replacing $s$ by $s\nu^{-1}$ if necessary, we may assume that $\pa \alpha =
\kappa = \delta s.$

Consider the transgression of $\alpha$. This involves a choice of $\beta
\in \vC_0^{k-1}(\cI M; A)$ such that $\delta \beta = \varepsilon^* \pa \alpha =
\varepsilon^* \kappa$ but there is a natural choice available. Namely, the
section $s$ of $P_k$ lifts canonically to a section of the trivial $A$-bundle
over $\Gamma^{(k)},$ from which $P_k$ is descended, and so defines a cochain
\[
	\wt s \in \vC_0^{k-1}(\cI M; A), \quad
	\wt s (\gamma) = a \iff s\bpns{\varepsilon(\gamma)} = [(\gamma,a)] \in P_k.
\]
That $\wt s$ is trivial on constant paths is a consequence of the fact that the
fusion condition implies that the descent data $\omega$ for $P_k$ is
trivial on constant loops. Since $\delta P_k$ is trivially descended from
the trivial bundle over $\Gamma^{(k+1)}$,
\[
\delta \wt s = \wt{\pns{\delta s}} = \varepsilon^\ast \delta s = \varepsilon^*\kappa,
\]
and hence $\beta = \wt s \in \vC_0^{k-1}(\cI M ; A)$ is an element such that
$\delta \beta = \varepsilon^* \kappa.$ It then follows that $d\beta
= \varrho_1^\ast \wt s^{-1}\,\varrho_2^\ast \wt s \equiv \omega \in \vCffe^{k-1}(\cL M; A)$
since
\[
\begin{gathered}
\wt s(\gamma)\,\wt s(\gamma')^{-1} = a\,{a'}^{-1}, \quad \text{such that} \\
s\bpns{\varepsilon(\gamma)} = s\bpns{\varepsilon(\gamma')} =
[(\gamma, a)] = [(\gamma',a')] \iff a = \omega(\gamma,\gamma')\,a'.
\end{gathered}
\]

In the other direction, fix a cocycle $\alpha \in \vC^{k}(M^2; A)$ and let
$\omega \in \vCffe^{k-1}(\cL M; A)$ represent $T[\alpha]^{-1}$, given by
$\omega = d \beta$ where $\delta \beta = \varepsilon^*\pa \alpha \in
\vC_0^{k}(\cI M; A).$ The regression of $\omega$ involves a
choice, of section of the bundle $P_k$, but here too there is a
natural one which recovers $\pa \alpha \in \vC^k(M^2; A)$.  Indeed, since
$\omega = \varrho_1^\ast \beta^{-1}\,\varrho_2^\ast \beta$, the equivalence
relation defining $P_k$ takes the particular form
\[
P_k \ni [(\gamma, a)] = [(\gamma', a')] \iff a = \beta(\gamma)\beta(\gamma')^{-1} a',
\]
and an appropriate section of $P_k$ is defined by
\[
	s(m,m') = [(\gamma,\beta(\gamma))] = [(\gamma', \beta(\gamma'))],
\]
since this equivalence class is independent of the particular $\gamma \in
\varepsilon^{-1}(m,m').$ With $s$ so defined, it follows that $\delta s \in
\vC^k(M; A)$ is given by
\[
\delta s(m,m') = [(\gamma, \delta \beta(\gamma)] 
= [(\gamma, \varepsilon^*\pa \alpha(\gamma))] = \pa \alpha(m,m'). \qedhere
\]
\end{proof}

\subsection{Compatibility} \label{S:compatibility}

The commutativity of the diagram \eqref{E:TRDiag} asserts that the `enhanced
transgression' map constructed above is compatible with transgression in the
usual sense. The latter corresponds to pullback of cohomology under the
evaluation map followed by projection onto the second factor under the
decomposition for the product: 
\begin{multline}
\ev^*:H^k(M;A)\to H^k(\bbS\times\cL M;A)\\
=H^k(\cL M;A)\oplus H^{k-1}(\cL M;A) \to H^{k-1}(\cL M;A).
\label{E:RegTran}
\end{multline}

To realize this in \v Cech cohomology, fix a small parameter $\delta >0$
and consider the open cover $\cS = \bigsqcup_{(t,l) \in \bbS \times \cL M} S_{t,l}$ of $\bbS\times \cL M$,
where
\begin{equation}
\begin{gathered}
S_{t,l}=\set{(t',l')\in\bbS\times\cL M : l'\in\Lambda_l,\ t' \in (t - \delta, t + \delta),\
l'(t')\in U_{l(t)}\ },\\
S_{t,l}\longrightarrow \Lambda_l,\quad S_{t,l}\longrightarrow I_{t}\subset\bbS,
\quad
\ev:S_{t,l}\ni(t',l')\longmapsto l'(t')\in U_{l(t)}.
\end{gathered}
\label{E:CechEval}
\end{equation}
The interval $I_t = (t-\delta,t+\delta)\subset\bbS$ is to be interpreted as the
`short' signed interval on $\bbS.$ This is a good cover, with respect to which
we consider the \v Cech complex on $\bbS\times \cL M.$ The evaluation map
$\ev : \bbS\times \cL M \to M$ and projections $\bbS \times \cL M \to \cL M$
and $\bbS \times \cL M \to \bbS$ lift to maps of the covers $\cS \to \cU$, $\cS
\to \Lambda$ and $\cS \to \cV$, respectively, where $\cV$ is the cover of
$\bbS$ by intervals of length $2\delta$ around each point. 

The first factor in the product \eqref{E:RegTran} corresponds to pullback to
$\cL M$ under the evaluation map at any fixed point on the circle.
Consequently, to consider the projection to the second factor of
\eqref{E:RegTran} we modify the pullback $\ev^*\alpha \in \vC^k(\bbS
\times \cL M; A)$ to
\begin{equation}
	\alpha' = (\ev_0^*\alpha)^{-1}\,\ev^*\alpha \in \vC^k(\bbS\times \cL M; A)
\label{E:ModPB}
\end{equation}
instead, where $\ev_0 : \bbS \times \cL M \ni (t,\ell)\mapsto \ell(0) \in M$
factors through the projection to $\cL M.$ Then the class of \eqref{E:ModPB}
projects to zero in $\vH^k(\cL M; A)$ and has the same projection as $\ev^*\alpha$
to $\vH^{k-1}(\cL M; A).$

To compute the latter, consider the space $[-1,1]\times \cL M$ which maps to
$\bbS\times\cL M$ by the identification of the endpoints. This has a good
cover $\cT = \bigsqcup_{t,l} T_{t,l}$ where $T_{t,l}$ is defined 
as in \eqref{E:CechEval} except that the interval is restricted to
$[-1,1].$ The map to $\bbS\times \cL M$ then lifts to a continuous map of
the covers. The image of \eqref{E:ModPB} lies in the subcomplex
$\vC_0^k([-1,1]\times \cL M; A)$ of chains which are trivial at $\set{0}\times
\cL M.$ This subcomplex is acyclic as in the proof of
Lemma~\ref{L:path_subcomplex} since $[-1,1]\times \cL M$ retracts onto $\set{0}\times
\cL M.$ Thus
\[
	\alpha' = \delta \sigma, \quad \sigma \in \vC^{k-1}([-1,1]\times \cL M; A),
\]
and the transgression class is represented by the difference
\begin{equation}
\pns{\sigma \rst_{\set{1}\times \cL M}} \pns{\sigma^{-1} \rst_{\set{-1}\times \cL M}}
\in \vC^{k-1}(\cL M; A).
	\label{E:transgression_class}
\end{equation}
That this is a cocycle follows from the fact that its \v Cech differential is
the difference of $\alpha'$ at $1$ and $-1$ which is trivial since
$\alpha'$ is pulled back from the circle.

On the other hand, the initial portion of the enhanced transgression
construction in \S\ref{S:transgression} may be modified as follows. Consider
the pullback
\[
\begin{gathered}
\wt \varepsilon^* \pa \alpha \in \vC^k([0,1]\times \cI M; A), \\
\wt \varepsilon : [0,1]\times \cI M \to M^2, \quad \wt \varepsilon(t,\gamma) =
\big(\gamma(0),\gamma(t)\big).
\end{gathered}
\]
As before this lies in an exact subcomplex, so $\wt \varepsilon^* \pa
\alpha = \delta \wt\beta$ where $\wt\beta \in \vC^{k-1}([0,1]\times \cI M; A),$
the restriction $\beta = \wt\beta \rst_{\set{1} \times \cI M}$ to a
cochain on $\cI M$ reduces to the earlier construction and $\wt \beta
\rst_{\set{0} \times \cI M}$ is trivial. Then the product
\[
\begin{gathered}
\sigma = \varsigma_1^* \wt \beta\,\varsigma_2^* \wt \beta \in
\vC^{k-1}([-1,1]\times \cI^{[2]} M; A), \\
\varsigma_i : [-1,1]\times \cI^{[2]} M \to [0,1]\times \cI M,\\
\varsigma_1\big(t,(\gamma_1,\gamma_2)\big)  = (\max(0,t), \gamma_1),
\quad\varsigma_2\big(t,(\gamma_1,\gamma_2)\big) = (-\min(0,t),\gamma_2)
\end{gathered}
\]
is a cochain on $[-1,1]\times \cL M$ with differential equal to $\alpha'$. Indeed, 
\begin{multline*}
\delta \sigma \big(t,\ell) 
= (\varsigma_1^*\delta \wt\beta\varsigma_2^*\delta \wt \beta)\big(t,(\gamma_1,\gamma_2)\big) 
= \begin{cases} \alpha(\gamma_1(t))\alpha^{-1}(\gamma_1(0)), & 0 \leq t \leq 1\\
	\alpha(\gamma_2(-t))\alpha^{-1}(\gamma_2(0)), & -1 \leq t \leq 0\end{cases}
\\= \alpha\big(\ell(t)\big)\alpha^{-1}\big(\ell(0)\big),
\end{multline*}
where $\ell = \psi(\gamma_1,\gamma_2).$ Finally, observe that the transgression
class \eqref{E:transgression_class} is represented by the `enhanced transgression'
class $d \beta^{-1}$:
\[
\pns{\sigma \rst_{\set{1}\times \cL M}}
\pns{\sigma^{-1} \rst_{\set{-1}\times \cL M}}(\gamma_1,\gamma_2)
= \wt \beta(1,\gamma_1)\,\wt \beta^{-1}(1,\gamma_2) = d \beta^{-1}(\gamma_1,\gamma_2).
\]
This completes the proof of the Theorem.

\bibliography{Loop-Fusion}
\bibliographystyle{amsalpha}

\end{document}